\documentclass{amsart}
\usepackage{amsmath}
\usepackage{amssymb}
\usepackage{amsfonts}

\newtheorem{theorem}{Theorem}
\theoremstyle{plain}

\newtheorem{corollary}{Corollary}

\newtheorem{definition}{Definition}

\numberwithin{equation}{section}
\input{tcilatex}

\begin{document}
\title[ Boundedness of Riesz Potential]{Some Inequalities for Riesz
Potential on Homogeneous Variable Exponent Herz-Morrey-Hardy Spaces}
\author{FER\.{I}T G\"{U}RB\"{U}Z}
\address{Department of Mathematics, K\i rklareli University, K\i rklareli
39100, T\"{u}rkiye }
\email{feritgurbuz@klu.edu.tr}
\urladdr{}
\thanks{}
\curraddr{ }
\urladdr{}
\thanks{}
\date{}
\dedicatory{}
\thanks{}

\begin{abstract}
In harmonic analysis, studies of inequalities of Riesz potential in various
function spaces have a very important place. Variable exponent Morrey type
spaces and the examines of the boundedness of such operators on these spaces
have an important place in harmonic analysis and have become an interesting
field. In this work, we obtain the boundedness of Riesz potential on
homogeneous variable exponent Herz-Morrey-Hardy spaces under some conditions.
\end{abstract}

\maketitle

\section{Introduction}

This work consists of 3 parts:

$\cdot $ Introduction (Studies conducted in this field have been given.)

$\cdot $ Definitions and Some Properties (General Notations, variable
exponent, Lebesgue space with variable exponent, homogeneous variable
exponent Morrey-Herz spaces, homogeneous variable exponent Herz-Morrey-Hardy
spaces, etc.have been given.)

$\cdot $ Main Result (Under some conditions, the boundedness of Riesz
potential on homogeneous variable exponent Herz-Morrey-Hardy spaces has been
given.)

Over the past twenty years, it has been clear that many contemporary issues
that arise in several mathematical models of applied sciences cannot be
studied using classical function spaces. New function spaces have to be
introduced and investigated as a result. Variable exponent Lebesgue and
Sobolev spaces, grand function spaces, variable exponent Herz spaces,
variable exponent Herz-Morrey type spaces, amalgam spaces and their hybrid
variations are examples of such spaces. Local regularity is better described
by Morrey spaces than by Lebesgue spaces. Consequently, Morrey spaces have
widespread use in both harmonic analysis and partial differential equation
(PDE) theory. After the fundamental work by Kov\'{a}\v{c}ik ve R\'{a}kosn%
\'{\i}k \cite{Kovacik} in 1991, there was a substantial advancement in the
theory of variable exponent function spaces.

Since Fefferman and Stein \cite{Fefferman and Stein} gave a real variable
characterization of Hardy space, the theory of function spaces has been an
important part of harmonic analysis. Many important achievements have been
made around Hardy space and its characterization as the main representative.
In 1968, Herz \cite{Herz} gave a new class of space are called the Herz
spaces, which characterized certain properties of some classical functions.
Actually, Herz spaces is the Lebesgue spaces with power weights $\left \vert
x\right \vert ^{\alpha }$. In recent years, with the discussion of variable
index function space, the results of Paluszynski \cite{Palus} have also been
generalized to variable index function spaces. Herz-type Hardy space is one
of them. The boundedness problem of Herz-type space and many operators on it
has also been rapidly developed \cite{Xu}-\cite{Babar}.

In recent years, many Herz type spaces have emerged. In this work, only
Herz-Morrey-Hardy spaces (which will be defined in the next section) will be
discussed. In this context, Herz-Morrey-Hardy spaces are the generalized
version of Herz-Hardy spaces. We will obtain some new results in these
spaces. We will consider the boundedness of some integral operator on
Herz-Morrey-Hardy spaces. In 2015, Xu and Yang \cite{Xu} first introduced
Herz-Morrey-Hardy spaces with variable exponents, secondly gave some results
concerning the boundedness of some singular integral operators on these
spaces. In 2016, Xu and Yang \cite{XU} gave the molecular decomposition of
variable exponent Herz-Morrey-Hardy spaces.

We wonder if these results \cite{Xu}\ can be generalized? In other words,
what properties do Riesz potential operators under size conditions provide
on variable exponent Herz-Morrey-Hardy spaces? Unfortunately, the use of
more general operators instead of classical operators, in other words the
use of the fractional operators and the boundedness of these operators on
these spaces have never been studied. With this work, it is planned to fill
this gap in the literature.

\section{Definitions and Some Properties}

Before giving the main result of this work, we first recall some elementary
facts, notations and necessary definitions that will be needed. We use

${\mathbb{R}^{n}=}\left \{ \left( x_{1},x_{2},\ldots ,x_{n}\right) :\text{
each }x_{i}\text{ is in }%
\mathbb{R}
\right \} $ to denote the $n$-dimensional Euclidean space. An important
class of functions consists of the \textbf{characteristic functions} of
sets. If $E $ is a set we define: $\chi _{E}\left( x\right) =\left \{ 
\begin{array}{cc}
1 & \text{if }x\in E, \\ 
0 & \text{if }x\notin E.%
\end{array}%
\right. .$ Denote by $\left \vert E\right \vert =\dint \limits_{E}dx$ the
Lebesgue measure of the set $E\subset \mathbb{%
\mathbb{R}
}{^{n}}$ and the characteristic function of a set $E$ is denoted by $\chi
_{E}$, hence $\chi _{E}\left( x\right) =1$ if $x\in E$ and zero otherwise. $%
E^{C}$ will always denote the complement of $E$. In this work, we can mainly
focus on the Riesz potential $I^{\beta }$ is defined by 
\begin{equation}
I^{\beta }f(x)=\int \limits_{%
\mathbb{R}
^{n}}\frac{f(y)}{|x-y|^{n-\beta }}dy\qquad 0<\beta <n..  \label{e1}
\end{equation}

Recently, authors have shown an increasing interest in the study of Riesz
potential. For example, $I^{\beta }$ is a bounded operator on a variable
exponent Lebesgue spaces $L^{p\left( \cdot \right) }\left( {\mathbb{R}^{n}}%
\right) $, that is, $I^{\beta }$ is bounded from $L^{p_{1}\left( \cdot
\right) }\left( {\mathbb{R}^{n}}\right) $ to $L^{p_{2}\left( \cdot \right)
}\left( {\mathbb{R}^{n}}\right) $ such that $1<p_{1}\left( \cdot \right)
<p_{2}\left( \cdot \right) <\infty $ and $\frac{1}{p_{2}\left( \cdot \right) 
}=\frac{\beta }{n}-\frac{1}{p_{1}\left( \cdot \right) }>0$ (see \cite{Capone}%
). Now, the definitions and lemmas which we are going to utilize to prove
the main result. Define $\mathcal{P}\left( {\mathbb{R}^{n}}\right) $ to be
the set of $p\left( \cdot \right) :{\mathbb{R}^{n}}\rightarrow \left[
1,\infty \right) $ when $1\leq p_{-}:=\limfunc{essinf}\limits_{x\in {\mathbb{%
R}^{n}}}p\left( x\right) $ and $p_{+}:=\limfunc{esssup}\limits_{x\in {%
\mathbb{R}^{n}}}p\left( x\right) <\infty $. The exponent $p^{\prime }\left(
\cdot \right) $ means the conjugate of $p\left( \cdot \right) $, that is $%
\frac{1}{p\left( x\right) }+\frac{1}{p^{\prime }\left( x\right) }=1$ holds.
We also claim that $C$ stands for a positive constant that can change its
value in each statement without explicit mention.

Now, we first define variable exponent Lebesgue space $L^{p\left( \cdot
\right) }$.

\begin{definition}
Let $p\left( \cdot \right) \in \mathcal{P}\left( {\mathbb{R}^{n}}\right) $.
Variable exponent Lebesgue space $L^{p\left( \cdot \right) }\left( {\mathbb{R%
}^{n}}\right) $ is defined by%
\begin{equation*}
\left \Vert f\right \Vert _{L^{p\left( \cdot \right) }\left( {\mathbb{R}^{n}}%
\right) }:=\inf \left \{ \eta >0:\dint \limits_{{\mathbb{R}^{n}}}\left( 
\frac{\left \vert f\left( x\right) \right \vert }{\eta }\right) ^{p\left(
x\right) }dx\leq 1\right \} <\infty .
\end{equation*}
\end{definition}

Let $f\in L_{loc}^{1}\left( {\mathbb{R}^{n}}\right) $. The Hardy-Littlewood
maximal operator $M$ is defined by%
\begin{equation*}
Mf\left( x\right) :=\sup_{r>0}\frac{1}{r^{n}}\int \limits_{B(x,r)}\left
\vert f\left( y\right) \right \vert dy,\qquad \forall x\in {\mathbb{R}^{n},}
\end{equation*}%
where and what follows $B(x,r)=\left \{ y\in {\mathbb{R}^{n}:}\left \vert
x-y\right \vert <r\right \} $ is the open ball centered at $x$ with radius $%
r $. $\mathcal{B}\left( {\mathbb{R}^{n}}\right) $ is the set of $p\left(
\cdot \right) \in \mathcal{P}\left( {\mathbb{R}^{n}}\right) $ satisfying the
condition that $M$ is bounded on $L^{p\left( \cdot \right) }\left( {\mathbb{R%
}^{n}}\right) $ as $p\left( \cdot \right) \in \mathcal{B}\left( {\mathbb{R}%
^{n}}\right) $ in \cite{Capone}.

\begin{definition}
Assume $\alpha \left( \cdot \right) $ is real-valued function on ${\mathbb{R}%
^{n}}$.

If 
\begin{equation}
\left \vert \alpha \left( x\right) -\alpha \left( y\right) \right \vert \leq 
\frac{C}{-\log \left( \left \vert x-y\right \vert \right) },\qquad \text{ }%
\left \vert x-y\right \vert \leq \frac{1}{2},\forall x,y\in \mathbb{%
\mathbb{R}
}^{n},C>0,  \label{f1}
\end{equation}%
then $\alpha \left( \cdot \right) $ is called local log-H\"{o}lder
continuous on ${\mathbb{R}^{n}}$.

If 
\begin{equation*}
\left \vert \alpha \left( x\right) -\alpha \left( 0\right) \right \vert \leq 
\frac{C}{\log \left( e+1/\left \vert x\right \vert \right) },\qquad \forall
x\in \mathbb{%
\mathbb{R}
}^{n},C>0,
\end{equation*}%
then $\alpha \left( \cdot \right) $ is called log-H\"{o}lder continuous at
origin.

If 
\begin{equation}
\left \vert \alpha \left( x\right) -\alpha _{\infty }\right \vert \leq \frac{%
C}{\log \left( e+\left \vert x\right \vert \right) },\qquad \alpha _{\infty
}\in 
\mathbb{R}
,C>0,  \label{f2}
\end{equation}%
then $\alpha \left( \cdot \right) $ is called log-H\"{o}lder continuous at
infinity.

If $\alpha \left( \cdot \right) $ satisfies both (\ref{f1}) and (\ref{f2}),
then $\alpha \left( \cdot \right) $ is called global log-H\"{o}lder
continuous. Also, the sets of functions above are denoted by $\mathcal{P}%
_{loc}^{\log }\left( 
\mathbb{R}
^{n}\right) $, $\mathcal{P}_{0}^{\log }\left( 
\mathbb{R}
^{n}\right) $, $\mathcal{P}_{\infty }^{\log }\left( 
\mathbb{R}
^{n}\right) $ and $\mathcal{P}^{\log }\left( 
\mathbb{R}
^{n}\right) $, respectively.
\end{definition}

Throughout this work, for simplicity, we denote $L^{p\left( \cdot \right)
}\left( {\mathbb{R}^{n}}\right) $ by $L^{p\left( \cdot \right) }$ and
similarly $B(x,r)$ by $B$. We will use the following results:

When $p\left( \cdot \right) \in \mathcal{B}\left( {\mathbb{R}^{n}}\right) $,
Izuki \cite{Izuki} proved that%
\begin{equation}
\frac{\left \Vert \chi _{S}\right \Vert _{L^{p\left( \cdot \right) }}}{\left
\Vert \chi _{B}\right \Vert _{L^{p\left( \cdot \right) }}}\leq C\left( \frac{%
\left \vert S\right \vert }{\left \vert B\right \vert }\right) ^{\delta
_{1}},\frac{\left \Vert \chi _{S}\right \Vert _{L^{p^{\prime }\left( \cdot
\right) }\left( 
\mathbb{R}
^{n}\right) }}{\left \Vert \chi _{B}\right \Vert _{L^{p^{\prime }\left(
\cdot \right) }\left( 
\mathbb{R}
^{n}\right) }}\leq C\left( \frac{\left \vert S\right \vert }{\left \vert
B\right \vert }\right) ^{\delta _{2}},S\subset B,  \label{1}
\end{equation}%
and{\large {\ }}%
\begin{equation}
\Vert \chi _{B}\Vert _{L^{p(\cdot )}}\Vert \chi _{B}\Vert _{L^{p^{\prime
}(\cdot )}}\leq C\left \vert B\right \vert ,  \label{5}
\end{equation}%
respectively.

By (\ref{1}) and (\ref{5}), we can deduce that 
\begin{equation}
\Vert \chi _{i}\Vert _{L^{q(\cdot )}}\leq \Vert \chi _{B_{i}}\Vert
_{L^{q(\cdot )}}.  \label{6}
\end{equation}

For $p(\cdot )\in \mathcal{P}(\mathbb{R}^{n})$, $f\in L^{p\left( \cdot
\right) }\left( {\mathbb{R}^{n}}\right) $ and $g\in L^{p^{\prime }\left(
\cdot \right) }\left( {\mathbb{R}^{n}}\right) $, H\"{o}lder inequality on
variable exponent Lebesgue spaces holds in the form 
\begin{equation}
\left \vert \dint \limits_{%
\mathbb{R}
^{n}}f\left( x\right) g\left( x\right) dx\right \vert \leq \dint \limits_{%
\mathbb{R}
^{n}}\left \vert f\left( x\right) g\left( x\right) \right \vert dx\leq
r_{p}\left \Vert f\right \Vert _{L^{p\left( \cdot \right) }\left( 
\mathbb{R}
^{n}\right) }\left \Vert g\right \Vert _{L^{p^{\prime }\left( \cdot \right)
}\left( 
\mathbb{R}
^{n}\right) }\qquad r_{p}=1+\frac{1}{p_{-}}-\frac{1}{p_{+}},  \label{3}
\end{equation}%
see Theorem 2.1 in \cite{Kovacik}.

Before giving the definition of the homogeneous variable exponent
Herz-Morrey spaces, let us introduce the following notations.

Let $l\in \mathbb{%
\mathbb{Z}
}$. Then, 
\begin{equation*}
\chi _{l}:=\chi _{F_{l}},F_{l}:=B_{l}\setminus
B_{l-1},B_{l}:=B(0,2^{l})=\{x\in \mathbb{R}^{n}:\left \vert x\right \vert
\leq 2^{l}\}.
\end{equation*}%
For $m\in 
\mathbb{N}
_{0}=%
\mathbb{N}
\cup \left \{ 0\right \} $, we define%
\begin{equation*}
\tilde{\chi}_{m}:=\left \{ 
\begin{array}{ccc}
\chi _{F_{m}} & , & m\geq 1 \\ 
\chi _{B_{0}} & , & m=0%
\end{array}%
\right. ,
\end{equation*}%
where the symbol $%
\mathbb{N}
_{0}$ denotes the set of all nonnegative integers.

Now let's give the definition of the homogeneous variable exponent
Herz-Morrey spaces.

\begin{definition}
Let $q\in \left( 0,\infty \right] $, $p\left( \cdot \right) \in \mathcal{P}%
\left( {\mathbb{R}^{n}}\right) $ and $\lambda \in \left[ 0,\infty \right) $.
Let $\alpha \left( \cdot \right) $ also be a bounded real-valued measurable
function on $\mathbb{R}^{n}\left( \text{that is, }\alpha \left( \cdot
\right) \in L^{\infty }\left( \mathbb{R}^{n}\right) \right) $. Then, the
homogeneous variable exponent Herz-Morrey space is defined by%
\begin{equation*}
M\dot{K}_{p(\cdot ),\lambda }^{\alpha \left( \cdot \right) ,q}(\mathbb{R}%
^{n})=\left \{ f\in L_{loc}^{p(\cdot )}\left( \mathbb{R}^{n}\setminus \{0\}
\right) :\Vert f\Vert _{M\dot{K}_{p(\cdot ),\lambda }^{\alpha \left( \cdot
\right) ,q}(\mathbb{R}^{n})}<\infty \right \} ,
\end{equation*}%
where%
\begin{equation*}
\Vert f\Vert _{M\dot{K}_{p(\cdot ),\lambda }^{\alpha \left( \cdot \right)
,q}(\mathbb{R}^{n})}:=\sup_{L\in \mathbb{%
\mathbb{Z}
}}2^{-L\lambda }\left( \sum \limits_{l=-\infty }^{L}\Vert 2^{l\alpha \left(
\cdot \right) }f\chi _{l}\Vert _{L^{p(\cdot )}}^{q}\right) ^{\frac{1}{q}}.
\end{equation*}
\end{definition}

The following equivalence was used in \cite{Xu} and is also very important
for the proof of our main result.

Firstly, let $p\left( \cdot \right) \in \mathcal{P}\left( {\mathbb{R}^{n}}%
\right) $, $0<q\leq \infty $ and $0\leq \lambda <\infty $. If $\alpha \left(
\cdot \right) \in L^{\infty }\left( \mathbb{R}^{n}\right) \cap \mathcal{P}%
_{0}^{\log }\left( 
\mathbb{R}
^{n}\right) \cap \mathcal{P}_{\infty }^{\log }\left( 
\mathbb{R}
^{n}\right) $, then 
\begin{equation}
\left \Vert f\right \Vert _{M\dot{K}_{p(\cdot ),\lambda }^{\alpha \left(
\cdot \right) ,q}(\mathbb{R}^{n})}^{q}\approx \max \left \{ 
\begin{array}{c}
\sup \limits_{L\leq 0,L\in 
\mathbb{Z}
}2^{-L\lambda q}\left( \sum \limits_{l=-\infty }^{L}2^{lq\alpha \left(
0\right) }\left \Vert f\chi _{l}\right \Vert _{L^{p(\cdot )}}^{q}\right) ,
\\ 
\sup \limits_{L>0,L\in 
\mathbb{Z}
}\left[ 
\begin{array}{c}
2^{-L\lambda q}\left( \sum \limits_{l=-\infty }^{-1}2^{lq\alpha \left(
0\right) }\left \Vert f\chi _{l}\right \Vert _{L^{p(\cdot )}}^{q}\right) \\ 
+2^{-L\lambda q}\left( \sum \limits_{l=0}^{L}2^{lq\alpha \left( \infty
\right) }\left \Vert f\chi _{l}\right \Vert _{L^{p(\cdot )}}^{q}\right)%
\end{array}%
\right]%
\end{array}%
\right \} .  \label{2}
\end{equation}

Now before giving the definition of homogeneous variable exponent
Herz-Morrey-Hardy space $HM\dot{K}_{p(\cdot ),\lambda }^{\alpha \left( \cdot
\right) ,q}(\mathbb{R}^{n})$, we need to recall some notations.

We assume that $\mathcal{S}\left( 
\mathbb{R}
^{n}\right) $ be the Schwartz space of all rapidly decreasing infinitely
differentiable functions on $%
\mathbb{R}
^{n}$, and $\mathcal{S}^{\prime }\left( 
\mathbb{R}
^{n}\right) $ be the dual space of $\mathcal{S}\left( 
\mathbb{R}
^{n}\right) $. Assume $G_{N}f$ is the grand maximal function of $f$ defined
by%
\begin{equation*}
G_{N}f\left( x\right) :=\sup_{\phi \in \mathcal{A}_{N}}\left \vert \phi
_{\nabla }^{\ast }\left( f\right) \left( x\right) \right \vert ,\qquad x\in 
\mathbb{R}
^{n},
\end{equation*}%
where $\mathcal{A}_{N}$ and $\phi _{\nabla }^{\ast }$ are defined in \cite%
{Xu}, respectively.

In 2015, Xu and Yang \cite{Xu} introduced the homogeneous variable exponent
Herz-Morrey-Hardy spaces $HM\dot{K}_{p(\cdot ),\lambda }^{\alpha \left(
\cdot \right) ,q}(\mathbb{R}^{n})$ as follows:

\begin{definition}
Let $\alpha \left( \cdot \right) \in L^{\infty }\left( \mathbb{R}^{n}\right) 
$, $q\in \left( 0,\infty \right] $, $p\left( \cdot \right) \in \mathcal{P}%
\left( {\mathbb{R}^{n}}\right) $, $\lambda \in \left[ 0,\infty \right) $ and 
$N>n+1$. Then, the homogeneous variable exponent Herz-Morrey-Hardy space $HM%
\dot{K}_{p(\cdot ),\lambda }^{\alpha \left( \cdot \right) ,q}(\mathbb{R}%
^{n}) $ is defined by%
\begin{equation*}
HM\dot{K}_{p(\cdot ),\lambda }^{\alpha \left( \cdot \right) ,q}(\mathbb{R}%
^{n}):=\left \{ 
\begin{array}{c}
f\in \mathcal{S}^{\prime }\left( 
\mathbb{R}
^{n}\right) :\Vert f\Vert _{HM\dot{K}_{p(\cdot ),\lambda }^{\alpha \left(
\cdot \right) ,q}(\mathbb{R}^{n})}:= \\ 
\Vert G_{N}f\Vert _{M\dot{K}_{p(\cdot ),\lambda }^{\alpha \left( \cdot
\right) ,q}(\mathbb{R}^{n})}<\infty%
\end{array}%
\right \} .
\end{equation*}
\end{definition}

We can say that $G_{N}f\left( x\right) \leq CMf\left( x\right) \left(
C>0\right) $ for $x\in 
\mathbb{R}
^{n}$ (see Proposition in Page 57 in \cite{Stein}). This means that $G_{N}f$
satisfies (13) in Lemma 8 in \cite{Xu}.Thus, if $p\left( \cdot \right) \in 
\mathcal{B}\left( {\mathbb{R}^{n}}\right) $ and $-n\delta _{1}<\alpha \left(
0\right) \leq \alpha _{\infty }<n\delta _{2}$ with $\delta _{1}$, $\delta
_{2}\in \left( 0,1\right) $, then we can write 
\begin{equation*}
HM\dot{K}_{p(\cdot ),\lambda }^{\alpha \left( \cdot \right) ,q}(\mathbb{R}%
^{n})\cap L_{loc}^{p(\cdot )}\left( \mathbb{R}^{n}\setminus \{0\} \right) =M%
\dot{K}_{p(\cdot ),\lambda }^{\alpha \left( \cdot \right) ,q}(\mathbb{R}%
^{n}).
\end{equation*}%
Also, to the show the boundedness of Calder\'{o}n-Zygmund operators, the
authors \cite{Xu} studied atomic characterization of the space $HM\dot{K}%
_{p(\cdot ),\lambda }^{\alpha \left( \cdot \right) ,q}(\mathbb{R}^{n})$ for
the case $n\delta _{2}\leq \alpha \left( 0\right) ,$ $\alpha _{\infty
}<\infty $ in terms of central atom.

\begin{definition}
Let $p\left( \cdot \right) \in \mathcal{P}\left( \mathbb{R}^{n}\right) $, $%
\alpha \left( \cdot \right) \in L^{\infty }\left( \mathbb{R}^{n}\right) \cap 
\mathcal{P}_{0}^{\log }\left( 
\mathbb{R}
^{n}\right) \cap \mathcal{P}_{\infty }^{\log }\left( 
\mathbb{R}
^{n}\right) $, and nonnegative integer $s\geqslant \left[ \alpha
_{r}-n\delta _{2}\right] $; here $\alpha _{r}=\alpha (0)$, if $r<1$, and $%
\alpha _{r}=\alpha _{\infty }$, if $r\geqslant 1$, $n\delta _{2}\leq \alpha
_{r}<\infty $ and $\delta _{2}$ as in (\ref{1}).

$(i)$ If a function $a$ on $\mathbb{R}^{n}$satisfies%
\begin{eqnarray*}
\left( 1\right) \text{ supp }a &\subset &B\left( 0,2^{r}\right) \\
\left( 2\right) \text{ }\left \Vert a\right \Vert _{L^{p\left( \cdot \right)
}} &\leq &\left \vert B\left( 0,2^{r}\right) \right \vert ^{-\frac{\alpha
_{r}}{n}} \\
\left( 3\right) \text{ }\dint \limits_{%
\mathbb{R}
^{n}}a\left( x\right) x^{\beta }dx &=&0,\text{ }\left \vert \beta \right
\vert \leq s,
\end{eqnarray*}%
then it is called a central $(\alpha (\cdot ),p(\cdot ))$-atom.

$(ii)$ If a function $a$ on $\mathbb{R}^{n}$satisfies $\left( 2\right) $ and 
$\left( 3\right) $ above and condition given below:%
\begin{equation*}
\left( 1\right) \text{ supp }a\subset B\left( 0,2^{r}\right) ,\text{ }r\geq
1,
\end{equation*}%
then it is called a central $(\alpha (\cdot ),p(\cdot ))$-atom.of limited
type.
\end{definition}

Finally, throughout this paper, we use the notation $f\lesssim g$ if $f\leq
Cg\left( C>0\right) $. If $f\lesssim g$ and $g\lesssim f$, then we will
denote $f\approx g$.

\section{Main Result}

First of all, before giving the main theorem, we will give the following
theorem proved by Xu and Yang \cite{Xu}. This theorem is very necessary for
the proof of the main theorem.

\begin{theorem}
\label{Theorem}Let $0<q<\infty $, $p\left( \cdot \right) \in \mathcal{B}%
\left( {\mathbb{R}^{n}}\right) $, $0\leq \lambda <\infty $, and $\alpha
\left( \cdot \right) \in L^{\infty }\left( \mathbb{R}^{n}\right) \cap 
\mathcal{P}_{0}^{\log }\left( 
\mathbb{R}
^{n}\right) \cap \mathcal{P}_{\infty }^{\log }\left( 
\mathbb{R}
^{n}\right) $ such that $2\lambda \leq \alpha \left( \cdot \right) $, $%
n\delta _{2}<\alpha \left( 0\right) $, $\alpha _{\infty }<\infty $ with $%
\delta _{1}$, $\delta _{2}\in \left( 0,1\right) $ satisfying (\ref{1}).%
\begin{equation*}
f\in HM\dot{K}_{p(\cdot ),\lambda }^{\alpha \left( \cdot \right) ,q}(\mathbb{%
R}^{n})\Longleftrightarrow f=\dsum \limits_{l=-\infty }^{\infty }\lambda
_{l}a_{l}
\end{equation*}%
in the sense of $\mathcal{S}^{\prime }\left( 
\mathbb{R}
^{n}\right) $, where each $a_{l}$ is a central $(\alpha (\cdot ),p(\cdot ))$%
-atom with support contained in $B_{l}$ and $\sup \limits_{L\in 
\mathbb{Z}
}2^{-L\lambda }\sum \limits_{l=-\infty }^{L}\left \vert \lambda
_{l}\right
\vert ^{q}<\infty $. Moreover,%
\begin{equation*}
\Vert f\Vert _{HM\dot{K}_{p(\cdot ),\lambda }^{\alpha \left( \cdot \right)
,q}(\mathbb{R}^{n})}\approx \inf \sup \limits_{L\in 
\mathbb{Z}
}2^{-L\lambda }\left( \sum \limits_{l=-\infty }^{L}\left \vert \lambda
_{l}\right \vert ^{q}\right) ^{\frac{1}{q}}.
\end{equation*}
\end{theorem}

\begin{corollary}
By Theorem \ref{Theorem},take $a_{j}=a\cdot \chi _{j}=a\cdot \chi _{B_{j}}$
for each $j\in 
\mathbb{Z}
$, we have 
\begin{equation*}
f=\dsum \limits_{j=0}^{\infty }\lambda _{j}a=\dsum \limits_{j=-\infty
}^{\infty }\lambda _{j}a_{j},
\end{equation*}%
where each $a_{l}$ is a central $(\alpha (\cdot ),q_{1}(\cdot ))$-atom with
support contained in $B_{j}$, and%
\begin{equation*}
\Vert f\Vert _{HM\dot{K}_{p_{1}(\cdot ),\lambda }^{\alpha \left( \cdot
\right) ,q_{1}}(\mathbb{R}^{n})}\approx \inf \sup \limits_{L\in 
\mathbb{Z}
}2^{-L\lambda }\left( \sum \limits_{j=-\infty }^{L}\left \vert \lambda
_{j}\right \vert ^{q_{1}}\right) ^{\frac{1}{q_{1}}}.
\end{equation*}%
By (\ref{3}), for each $j,k\in 
\mathbb{Z}
$, with $k\leq L$ and $j\leq k-1$, then $\left \vert x-y\right \vert \sim
\left \vert x\right \vert $, $2\left \vert y\right \vert \leq \left \vert
x\right \vert $, we get%
\begin{eqnarray}
\left \vert I^{\beta }a_{j}\chi _{k}\left( x\right) \right \vert &=&\int
\limits_{%
\mathbb{R}
^{n}}\frac{\left \vert a_{j}\left( x\right) \right \vert }{|x-y|^{n-\beta }}%
dy\cdot \chi _{k}\left( x\right)  \notag \\
&\lesssim &2^{k\left( \beta -n\right) }\int \limits_{%
\mathbb{R}
^{n}}\left \vert a_{j}\left( x\right) \right \vert dy\cdot \chi _{k}\left(
x\right)  \notag \\
&\lesssim &2^{k\left( \beta -n\right) }\left \Vert a_{j}\right \Vert
_{L^{p_{1}(\cdot )}}\left \Vert \chi _{j}\right \Vert _{L^{p_{1}^{\prime
}(\cdot )}}\cdot \chi _{k}\left( x\right) .  \label{4}
\end{eqnarray}%
So applying (\ref{1}), (\ref{6}), (\ref{3}) and (\ref{4}), we have%
\begin{eqnarray}
\left \Vert \left( I^{\beta }a_{j}\right) \chi _{k}\right \Vert
_{L^{p_{2}(\cdot )}} &\lesssim &2^{k\left( \beta -n\right) }\left \Vert
a_{j}\right \Vert _{L^{p_{1}(\cdot )}}\left \Vert \chi _{j}\right \Vert
_{L^{p_{1}^{\prime }(\cdot )}}\left \Vert \chi _{k}\right \Vert
_{L^{p_{2}(\cdot )}}  \notag \\
&\lesssim &2^{k\beta }\left \Vert a_{j}\right \Vert _{L^{p_{1}(\cdot
)}}\left \Vert \chi _{j}\right \Vert _{L^{p_{1}^{\prime }(\cdot
)}}2^{-kn}\left \Vert \chi _{B_{k}}\right \Vert _{L^{p_{2}(\cdot )}}  \notag
\\
&\lesssim &2^{k\beta }\left \Vert a_{j}\right \Vert _{L^{p_{1}(\cdot
)}}\left \Vert \chi _{j}\right \Vert _{L^{p_{1}^{\prime }(\cdot )}}\left
\Vert \chi _{B_{k}}\right \Vert _{L^{p_{2}^{\prime }(\cdot )}}^{-1}  \notag
\\
&\lesssim &2^{k\beta }\left \Vert a_{j}\right \Vert _{L^{p_{1}(\cdot
)}}\left \Vert \chi _{j}\right \Vert _{L^{p_{1}^{\prime }(\cdot )}}\left
\Vert \chi _{B_{k}}\right \Vert _{L^{p_{2}^{\prime }(\cdot )}}^{-1}\frac{%
\left \Vert \chi _{B_{j}}\right \Vert _{L^{p_{2}^{\prime }(\cdot )}}}{\left
\Vert \chi _{B_{k}}\right \Vert _{L^{p_{2}^{\prime }(\cdot )}}}  \notag \\
&\lesssim &2^{k\beta }\left \Vert a_{j}\right \Vert _{L^{p_{1}(\cdot
)}}\left \Vert \chi _{j}\right \Vert _{L^{p_{1}^{\prime }(\cdot )}}\left
\Vert \chi _{B_{j}}\right \Vert _{L^{p_{2}^{\prime }(\cdot
)}}^{-1}2^{n\delta _{2}\left( j-k\right) }  \notag \\
&\lesssim &2^{\left( n\delta _{2}-\beta \right) \left( j-k\right) }\left
\Vert a_{j}\right \Vert _{L^{p_{1}(\cdot )}}\left( 2^{-jn}\left \Vert \chi
_{B_{j}}\right \Vert _{L^{p_{2}(\cdot )}}\left \Vert \chi _{B_{j}}\right
\Vert _{L^{p_{2}^{\prime }(\cdot )}}\right) ^{-1}  \notag \\
&\lesssim &2^{\left( \beta -n\delta _{2}\right) \left( k-j\right) }\left
\Vert a_{j}\right \Vert _{L^{p_{1}(\cdot )}}  \notag \\
&\lesssim &2^{\left( \beta -n\delta _{2}\right) \left( k-j\right) -\alpha
_{j}j},  \label{12}
\end{eqnarray}%
where in the sixth inequality we have used the following fact:

First, we know that 
\begin{equation*}
I^{\beta }\left( \chi _{B_{j}}\right) \left( x\right) \geq \int
\limits_{B_{j}}\frac{dy}{|x-y|^{n-\beta }}\cdot \chi _{B_{j}}\left( x\right)
\geq C2^{\beta j}\chi _{B_{j}}\left( x\right) .
\end{equation*}%
Also, let $p_{1}\left( \cdot \right) \in \mathcal{P}\left( {\mathbb{R}^{n}}%
\right) $.and $0<\beta <\frac{n}{\left( p_{1}\right) _{+}}$. Then, since $%
I^{\beta }$ is bounded from $L^{p_{1}\left( \cdot \right) }$ to $%
L^{p_{2}\left( \cdot \right) }$ and using (\ref{5}), we obtain%
\begin{eqnarray*}
\left \Vert \chi _{B_{j}}\right \Vert _{L^{p_{2}(\cdot )}} &\lesssim
&2^{-j\beta }\left \Vert I^{\beta }\left( \chi _{B_{j}}\right) \right \Vert
_{L^{p_{2}(\cdot )}} \\
&\lesssim &2^{-j\beta }\left \Vert \chi _{B_{j}}\right \Vert
_{L^{p_{1}(\cdot )}} \\
&\lesssim &2^{\left( n-\beta \right) j}\left \Vert \chi _{B_{j}}\right \Vert
_{L^{p_{1}^{\prime }(\cdot )}}^{-1} \\
&\lesssim &2^{\left( n-\beta \right) j}\left \Vert \chi _{j}\right \Vert
_{L^{p_{1}^{\prime }(\cdot )}}^{-1}.
\end{eqnarray*}%
On the other hand, for each $j,k\in 
\mathbb{Z}
$, with $k\leq L$ and $j\geq k-1$, then $\left \vert x-y\right \vert \sim
\left \vert x\right \vert $, $2\left \vert y\right \vert \leq \left \vert
x\right \vert $. Then, $\left \Vert \left( I^{\beta }a_{j}\right) \chi
_{k}\right \Vert _{L^{p_{2}(\cdot )}}$ has the same estimate above, here we
omit the details, thus the inequality 
\begin{equation}
\left \Vert \left( I^{\beta }a_{j}\right) \chi _{k}\right \Vert
_{L^{p_{2}(\cdot )}}\lesssim 2^{\left( \beta -n\delta _{1}\right) \left(
j-k\right) -\alpha _{j}j}  \label{0}
\end{equation}%
is valid.
\end{corollary}

The following theorem is our main result.

\begin{theorem}
Let $0<q_{1}\leq q_{2}<\infty $, $0<\lambda <\infty $, $0<\beta <n$, $%
0<p_{1}\leq p_{2}<\infty $, $p_{1}\left( \cdot \right) $, $p_{2}\left( \cdot
\right) \in \mathcal{B}\left( {\mathbb{R}^{n}}\right) $ with $\frac{1}{%
p_{1}\left( \cdot \right) }=\frac{1}{p_{2}\left( \cdot \right) }+\frac{\beta 
}{n}$ and $\alpha \left( \cdot \right) \in L^{\infty }\left( \mathbb{R}%
^{n}\right) \cap \mathcal{P}_{0}^{\log }\left( 
\mathbb{R}
^{n}\right) \cap \mathcal{P}_{\infty }^{\log }\left( 
\mathbb{R}
^{n}\right) $ such that $2\lambda \leq \alpha \left( \cdot \right) $, $\beta
-n\delta _{2}<\alpha \left( 0\right) $, $\alpha _{\infty }<\infty $ with $%
\delta _{1}$, $\delta _{2}\in \left( 0,1\right) $ satisfying (\ref{1}). Let
also $I^{\beta }$ be defined as in (\ref{e1}) and be bounded from $%
L^{p_{1}\left( \cdot \right) }\left( {\mathbb{R}^{n}}\right) $ to $%
L^{p_{2}\left( \cdot \right) }\left( {\mathbb{R}^{n}}\right) $ when $%
1<p_{1}\left( \cdot \right) <p_{2}\left( \cdot \right) <\infty $ and $\frac{1%
}{p_{2}\left( \cdot \right) }=\frac{\beta }{n}-\frac{1}{p_{1}\left( \cdot
\right) }>0$. Then for all $f\in HM\dot{K}_{p_{1}(\cdot ),\lambda }^{\alpha
\left( \cdot \right) ,q_{1}}(\mathbb{R}^{n})$, $I^{\beta }$ is bounded from $%
HM\dot{K}_{p_{1}(\cdot ),\lambda }^{\alpha \left( \cdot \right) ,q_{1}}(%
\mathbb{R}^{n})$ to $M\dot{K}_{p_{2}(\cdot ),\lambda }^{\alpha \left( \cdot
\right) ,q_{2}}(\mathbb{R}^{n})$. That is,%
\begin{equation}
\left \Vert I^{\beta }f\right \Vert _{M\dot{K}_{p_{2}(\cdot ),\lambda
}^{\alpha \left( \cdot \right) ,q_{2}}(\mathbb{R}^{n})}\lesssim \left \Vert
f\right \Vert _{HM\dot{K}_{p_{1}(\cdot ),\lambda }^{\alpha \left( \cdot
\right) ,q_{1}}(\mathbb{R}^{n})}.  \label{100}
\end{equation}
\end{theorem}

\begin{proof}
Assume $f\in HM\dot{K}_{p_{1}(\cdot ),\lambda }^{\alpha \left( \cdot \right)
,q_{1}}(\mathbb{R}^{n})$, $f=\dsum \limits_{j=-\infty }^{\infty }\lambda
_{j}a_{j}$ and%
\begin{equation*}
\Vert f\Vert _{HM\dot{K}_{p_{1}(\cdot ),\lambda }^{\alpha \left( \cdot
\right) ,q_{1}}(\mathbb{R}^{n})}\approx \inf \sup \limits_{L\in 
\mathbb{Z}
}2^{-L\lambda }\left( \sum \limits_{j=-\infty }^{L}\left \vert \lambda
_{j}\right \vert ^{q_{1}}\right) ^{\frac{1}{q_{1}}}.
\end{equation*}%
For convenience, we indicate 
\begin{equation*}
\Lambda =\sup \limits_{L\in 
\mathbb{Z}
}2^{-L\lambda }\sum \limits_{j=-\infty }^{L}\left \vert \lambda _{j}\right
\vert ^{q_{1}}.
\end{equation*}%
By (\ref{2}), we have 
\begin{eqnarray*}
\left \Vert I^{\beta }f\right \Vert _{M\dot{K}_{p_{2}(\cdot ),\lambda
}^{\alpha \left( \cdot \right) ,q_{2}}(\mathbb{R}^{n})}^{q_{1}} &\approx
&\max \left \{ 
\begin{array}{c}
\sup \limits_{L\leq 0,L\in 
\mathbb{Z}
}2^{-L\lambda q_{1}}\left( \sum \limits_{k=-\infty }^{L}2^{kq_{1}\alpha
\left( 0\right) }\left \Vert \left( I^{\beta }f\right) \chi _{k}\right \Vert
_{L^{p_{2}(\cdot )}}^{q_{1}}\right) , \\ 
\sup \limits_{L>0,L\in 
\mathbb{Z}
}\left[ 
\begin{array}{c}
2^{-L\lambda q_{1}}\left( \sum \limits_{k=-\infty }^{-1}2^{kq_{1}\alpha
\left( 0\right) }\left \Vert \left( I^{\beta }f\right) \chi _{k}\right \Vert
_{L^{p_{2}(\cdot )}}^{q_{1}}\right) \\ 
+2^{-L\lambda q_{1}}\left( \sum \limits_{k=0}^{L}2^{kq_{1}\alpha _{\infty
}}\left \Vert \left( I^{\beta }f\right) \chi _{k}\right \Vert
_{L^{p_{2}(\cdot )}}^{q_{1}}\right)%
\end{array}%
\right]%
\end{array}%
\right \} \\
&\lesssim &\left \{ F,G+H\right \} ,
\end{eqnarray*}%
where%
\begin{eqnarray*}
F &:&=\sup \limits_{L\leq 0,L\in 
\mathbb{Z}
}2^{-L\lambda q_{1}}\left( \sum \limits_{k=-\infty }^{L}2^{kq_{1}\alpha
\left( 0\right) }\left \Vert \left( I^{\beta }f\right) \chi _{k}\right \Vert
_{L^{p_{2}(\cdot )}}^{q_{1}}\right) \\
G &:&=\sum \limits_{k=-\infty }^{-1}2^{kq_{1}\alpha \left( 0\right) }\left
\Vert \left( I^{\beta }f\right) \chi _{k}\right \Vert _{L^{p_{2}(\cdot
)}}^{q_{1}} \\
H &:&=\sup \limits_{L>0,L\in 
\mathbb{Z}
}2^{-L\lambda q_{1}}\left( \sum \limits_{k=0}^{L}2^{kq_{1}\alpha _{\infty
}}\left \Vert \left( I^{\beta }f\right) \chi _{k}\right \Vert
_{L^{p_{2}(\cdot )}}^{q_{1}}\right) .
\end{eqnarray*}

We first estimate $F:$%
\begin{eqnarray*}
F &:&=\sup \limits_{L\leq 0,L\in 
\mathbb{Z}
}2^{-L\lambda q_{1}}\left( \sum \limits_{k=-\infty }^{L}2^{kq_{1}\alpha
\left( 0\right) }\left \Vert \left( I^{\beta }f\right) \chi _{k}\right \Vert
_{L^{p_{2}(\cdot )}}^{q_{1}}\right) \\
&\leq &\sup \limits_{L\leq 0,L\in 
\mathbb{Z}
}2^{-L\lambda q_{1}}\sum \limits_{k=-\infty }^{L}2^{kq_{1}\alpha \left(
0\right) }\left( \sum \limits_{j=k}^{\infty }\left \vert \lambda _{j}\right
\vert \left \Vert \left( I^{\beta }a_{j}\right) \chi _{k}\right \Vert
_{L^{p_{2}(\cdot )}}\right) ^{q_{1}} \\
&&+\sup \limits_{L\leq 0,L\in 
\mathbb{Z}
}2^{-L\lambda q_{1}}\sum \limits_{k=-\infty }^{L}2^{kq_{1}\alpha \left(
0\right) }\left( \sum \limits_{j=-\infty }^{k-1}\left \vert \lambda
_{j}\right \vert \left \Vert \left( I^{\beta }a_{j}\right) \chi _{k}\right
\Vert _{L^{p_{2}(\cdot )}}\right) ^{q_{1}} \\
&:&=F_{1}+F_{2}.
\end{eqnarray*}%
Since $I^{\beta }$ is bounded from $L^{p_{1}\left( \cdot \right) }$ to $%
L^{p_{2}\left( \cdot \right) }$, then we have%
\begin{equation}
\left \Vert \left( I^{\beta }a_{j}\right) \chi _{k}\right \Vert
_{L^{p_{2}(\cdot )}}\leq \left \Vert a_{j}\right \Vert _{L^{p_{1}(\cdot
)}}\leq \left \vert B_{j}\right \vert ^{-\frac{\alpha _{j}}{n}}=2^{-j\alpha
_{j}}.  \label{10}
\end{equation}%
Hence, we consider $F_{1}$ in two cases.

\textbf{Case 1 }$\left( 0<q_{1}\leq 1\right) .$

By (\ref{10}) and the inequality%
\begin{equation}
\left( \dsum \limits_{j=1}^{\infty }c_{j}\right) ^{q_{1}}\leq \dsum
\limits_{j=1}^{\infty }c_{j}^{q_{1}}\text{, }c_{j}\geq 0,\text{ }j\in 
\mathbb{N}
,  \label{11}
\end{equation}%
we have%
\begin{eqnarray*}
F_{1} &=&\sup \limits_{L\leq 0,L\in 
\mathbb{Z}
}2^{-L\lambda q_{1}}\sum \limits_{k=-\infty }^{L}2^{kq_{1}\alpha \left(
0\right) }\left( \sum \limits_{j=k}^{\infty }\left \vert \lambda _{j}\right
\vert \left \Vert \left( I^{\beta }a_{j}\right) \chi _{k}\right \Vert
_{L^{p_{2}(\cdot )}}\right) ^{q_{1}} \\
&\lesssim &\sup \limits_{L\leq 0,L\in 
\mathbb{Z}
}2^{-L\lambda q_{1}}\sum \limits_{k=-\infty }^{L}2^{kq_{1}\alpha \left(
0\right) }\left( \sum \limits_{j=k}^{\infty }\left \vert \lambda _{j}\right
\vert 2^{-j\alpha _{j}}\right) ^{q_{1}} \\
&\lesssim &\sup \limits_{L\leq 0,L\in 
\mathbb{Z}
}2^{-L\lambda q_{1}}\sum \limits_{k=-\infty }^{L}2^{kq_{1}\alpha \left(
0\right) }\left( \sum \limits_{j=k}^{-1}\left \vert \lambda _{j}\right \vert
^{q_{1}}2^{-j\alpha \left( 0\right) q_{1}}+\sum \limits_{j=0}^{\infty }\left
\vert \lambda _{j}\right \vert ^{q_{1}}2^{-j\alpha _{\infty }q_{1}}\right)
\end{eqnarray*}%
\begin{eqnarray*}
&\lesssim &\sup \limits_{L\leq 0,L\in 
\mathbb{Z}
}2^{-L\lambda q_{1}}\sum \limits_{k=-\infty }^{L}\sum
\limits_{j=k}^{-1}\left \vert \lambda _{j}\right \vert ^{q_{1}}2^{\left(
k-j\right) \alpha \left( 0\right) q_{1}} \\
&&+\sup \limits_{L\leq 0,L\in 
\mathbb{Z}
}2^{-L\lambda q_{1}}\sum \limits_{k=-\infty }^{L}2^{k\alpha \left( 0\right)
q_{1}}\sum \limits_{j=0}^{\infty }\left \vert \lambda _{j}\right \vert
^{q_{1}}2^{-j\alpha _{\infty }q_{1}} \\
&\lesssim &\sup \limits_{L\leq 0,L\in 
\mathbb{Z}
}2^{-L\lambda q_{1}}\sum \limits_{j=-\infty }^{-1}\left \vert \lambda
_{j}\right \vert ^{q_{1}}\sum \limits_{k=-\infty }^{j}2^{\left( k-j\right)
\alpha \left( 0\right) q_{1}} \\
&&+\sup \limits_{L\leq 0,L\in 
\mathbb{Z}
}\sum \limits_{j=0}^{\infty }2^{-j\lambda q_{1}}\left \vert \lambda
_{j}\right \vert ^{q_{1}}2^{\left( \lambda -\alpha _{\infty }\right)
jq_{1}}2^{-L\lambda q_{1}}\sum \limits_{k=-\infty }^{L}2^{k\alpha \left(
0\right) q_{1}} \\
&\leq &\sup \limits_{L\leq 0,L\in 
\mathbb{Z}
}2^{-L\lambda q_{1}}\sum \limits_{j=-\infty }^{L}\left \vert \lambda
_{j}\right \vert ^{q_{1}}+\sup \limits_{L\leq 0,L\in 
\mathbb{Z}
}2^{-L\lambda q_{1}}\sum \limits_{j=L}^{-1}\left \vert \lambda _{j}\right
\vert ^{q_{1}}\sum \limits_{k=-\infty }^{j}2^{\left( k-j\right) \alpha
\left( 0\right) q_{1}} \\
&&+\Lambda \sup \limits_{L\leq 0,L\in 
\mathbb{Z}
}\sum \limits_{j=0}^{\infty }2^{\left( \lambda -\alpha _{\infty }\right)
jq_{1}}\sum \limits_{k=-\infty }^{L}2^{\left( \alpha \left( 0\right)
k-L\lambda \right) q_{1}} \\
&\lesssim &\Lambda +\sup \limits_{L\leq 0,L\in 
\mathbb{Z}
}\sum \limits_{j=L}^{-1}2^{-j\lambda q_{1}}\left \vert \lambda _{j}\right
\vert ^{q_{1}}2^{\left( j-L\right) \lambda q_{1}}\sum \limits_{k=-\infty
}^{j}2^{\left( k-j\right) \alpha \left( 0\right) q_{1}}+\Lambda \\
&\lesssim &\Lambda +\Lambda \sup \limits_{L\leq 0,L\in 
\mathbb{Z}
}\sum \limits_{j=L}^{-1}2^{\left( j-L\right) \lambda q_{1}}\sum
\limits_{k=-\infty }^{j}2^{\left( k-j\right) \alpha \left( 0\right) q_{1}} \\
&\lesssim &\Lambda \text{ }\left( \alpha _{\infty }>\lambda \right) .
\end{eqnarray*}%
\textbf{Case 2 }$\left( 1<q_{1}<\infty \right) .$

In this case, we use H\"{o}lder's inequality instead of (\ref{11}). Let $%
\frac{1}{q_{1}}+\frac{1}{q_{1}^{\prime }}=1$. By (\ref{10}) and H\"{o}lder's
inequality,%
\begin{eqnarray*}
F_{1} &\lesssim &\sup \limits_{L\leq 0,L\in 
\mathbb{Z}
}2^{-L\lambda q_{1}}\sum \limits_{k=-\infty }^{L}2^{kq_{1}\alpha \left(
0\right) }\left( \sum \limits_{j=k}^{\infty }\left \vert \lambda _{j}\right
\vert 2^{-j\alpha _{j}}\right) ^{q_{1}} \\
&\lesssim &\sup \limits_{L\leq 0,L\in 
\mathbb{Z}
}2^{-L\lambda q_{1}}\sum \limits_{k=-\infty }^{L}\left( \sum
\limits_{j=k}^{-1}\left \vert \lambda _{j}\right \vert 2^{\left( k-j\right)
\alpha \left( 0\right) }\right) ^{q_{1}} \\
&&+\sup \limits_{L\leq 0,L\in 
\mathbb{Z}
}2^{-L\lambda q_{1}}\sum \limits_{k=-\infty }^{L}2^{k\alpha \left( 0\right)
q_{1}}\left( \sum \limits_{j=0}^{\infty }\left \vert \lambda _{j}\right
\vert 2^{-j\alpha _{\infty }}\right) ^{q_{1}} \\
&\lesssim &\sup \limits_{L\leq 0,L\in 
\mathbb{Z}
}2^{-L\lambda q_{1}}\sum \limits_{k=-\infty }^{L}\left( \sum
\limits_{j=k}^{-1}\left \vert \lambda _{j}\right \vert ^{q_{1}}2^{\frac{%
\left( k-j\right) \alpha \left( 0\right) q_{1}}{2}}\right) \left( \sum
\limits_{j=k}^{-1}2^{\frac{\left( k-j\right) \alpha \left( 0\right)
q_{1}^{\prime }}{2}}\right) ^{\frac{q_{1}}{q_{1}^{\prime }}} \\
&&+\sup \limits_{L\leq 0,L\in 
\mathbb{Z}
}2^{-L\lambda q_{1}}\sum \limits_{k=-\infty }^{L}2^{kq_{1}\alpha \left(
0\right) }\left( \sum \limits_{j=0}^{\infty }\left \vert \lambda _{j}\right
\vert ^{q_{1}}2^{\frac{-j\alpha _{\infty }q_{1}}{2}}\right) \left( \sum
\limits_{j=0}^{\infty }2^{\frac{-j\alpha _{\infty }q_{1}^{\prime }}{2}%
}\right) ^{\frac{q_{1}}{q_{1}^{\prime }}}
\end{eqnarray*}%
\begin{eqnarray*}
&\lesssim &\sup \limits_{L\leq 0,L\in 
\mathbb{Z}
}2^{-L\lambda q_{1}}\sum \limits_{k=-\infty }^{L}\sum
\limits_{j=k}^{-1}\left \vert \lambda _{j}\right \vert ^{q_{1}}2^{\frac{%
\left( k-j\right) \alpha \left( 0\right) q_{1}}{2}} \\
&&+\sup \limits_{L\leq 0,L\in 
\mathbb{Z}
}2^{-L\lambda q_{1}}\sum \limits_{k=-\infty }^{L}2^{kq_{1}\alpha \left(
0\right) }\sum \limits_{j=0}^{\infty }\left \vert \lambda _{j}\right \vert
^{q_{1}}2^{\frac{-j\alpha _{\infty }q_{1}}{2}} \\
&\lesssim &\sup \limits_{L\leq 0,L\in 
\mathbb{Z}
}2^{-L\lambda q_{1}}\sum \limits_{j=-\infty }^{-1}\left \vert \lambda
_{j}\right \vert ^{q_{1}}\sum \limits_{k=-\infty }^{j}2^{\frac{\left(
k-j\right) \alpha \left( 0\right) q_{1}}{2}} \\
&&+\sup \limits_{L\leq 0,L\in 
\mathbb{Z}
}\sum \limits_{j=0}^{\infty }2^{-j\lambda q_{1}}\left \vert \lambda
_{j}\right \vert ^{q_{1}}2^{\left( \lambda -\frac{\alpha _{\infty }}{2}%
\right) jq_{1}}2^{-L\lambda q_{1}}\sum \limits_{k=-\infty }^{L}2^{k\alpha
\left( 0\right) q_{1}} \\
&\leq &\sup \limits_{L\leq 0,L\in 
\mathbb{Z}
}2^{-L\lambda q_{1}}\sum \limits_{j=-\infty }^{L}\left \vert \lambda
_{j}\right \vert ^{q_{1}}+\sup \limits_{L\leq 0,L\in 
\mathbb{Z}
}2^{-L\lambda q_{1}}\sum \limits_{j=L}^{-1}\left \vert \lambda _{j}\right
\vert ^{q_{1}}\sum \limits_{k=-\infty }^{j}2^{\frac{\left( k-j\right) \alpha
\left( 0\right) q_{1}}{2}} \\
&&+\Lambda \sup \limits_{L\leq 0,L\in 
\mathbb{Z}
}\sum \limits_{j=0}^{\infty }2^{\left( \lambda -\frac{\alpha _{\infty }}{2}%
\right) jq_{1}}\sum \limits_{k=-\infty }^{L}2^{\left( \alpha \left( 0\right)
k-L\lambda \right) q_{1}} \\
&\lesssim &\Lambda +\sup \limits_{L\leq 0,L\in 
\mathbb{Z}
}\sum \limits_{j=L}^{-1}2^{-j\lambda q_{1}}\left \vert \lambda _{j}\right
\vert ^{q_{1}}2^{\left( j-L\right) \lambda q_{1}}\sum \limits_{k=-\infty
}^{j}2^{\frac{\left( k-j\right) \alpha \left( 0\right) q_{1}}{2}}+\Lambda \\
&\lesssim &\Lambda +\Lambda \sup \limits_{L\leq 0,L\in 
\mathbb{Z}
}\sum \limits_{j=L}^{-1}2^{\left( j-L\right) \lambda q_{1}}\sum
\limits_{k=-\infty }^{j}2^{\frac{\left( k-j\right) \alpha \left( 0\right)
q_{1}}{2}} \\
&\lesssim &\Lambda \text{ }\left( \alpha _{\infty }>2\lambda \right) .
\end{eqnarray*}%
Thus, we have $F_{1}\lesssim \Lambda $.

Second, we estimate $F_{2}$. We examine $F_{2}$ into two cases $0<q_{1}\leq
1 $ and $1<q_{1}<\infty $.

Let $0<q_{1}\leq 1$. By (\ref{11}), (\ref{12}) and the hypothesis $\beta
-n\delta _{2}<\alpha \left( 0\right) $, we get 
\begin{eqnarray*}
F_{2} &=&\sup \limits_{L\leq 0,L\in 
\mathbb{Z}
}2^{-L\lambda q_{1}}\sum \limits_{k=-\infty }^{L}2^{kq_{1}\alpha \left(
0\right) }\left( \sum \limits_{j=-\infty }^{k-1}\left \vert \lambda
_{j}\right \vert \left \Vert \left( I^{\beta }a_{j}\right) \chi _{k}\right
\Vert _{L^{p_{2}(\cdot )}}\right) ^{q_{1}} \\
&\lesssim &\sup \limits_{L\leq 0,L\in 
\mathbb{Z}
}2^{-L\lambda q_{1}}\sum \limits_{k=-\infty }^{L}2^{kq_{1}\alpha \left(
0\right) }\left( \sum \limits_{j=-\infty }^{k-1}\left \vert \lambda
_{j}\right \vert ^{q_{1}}2^{\left[ \left( \beta -n\delta _{2}\right) \left(
k-j\right) -\alpha \left( 0\right) j\right] q_{1}}\right) \\
&\lesssim &\sup \limits_{L\leq 0,L\in 
\mathbb{Z}
}2^{-L\lambda q_{1}}\sum \limits_{j=-\infty }^{L}\left \vert \lambda
_{j}\right \vert ^{q_{1}}\sum \limits_{k=j+1}^{-1}2^{\left( \beta -n\delta
_{2}-\alpha \left( 0\right) \right) \left( k-j\right) q_{1}} \\
&\lesssim &\Lambda .
\end{eqnarray*}%
Now, let $1<q_{1}<\infty $. and $\frac{1}{q_{1}}+\frac{1}{q_{1}^{\prime }}=1$%
. By (\ref{12}), H\"{o}lder's inequality and the assumption $\beta -n\delta
_{2}<\alpha \left( 0\right) $,%
\begin{equation*}
F_{2}\lesssim \sup \limits_{L\leq 0,L\in 
\mathbb{Z}
}2^{-L\lambda q_{1}}\sum \limits_{k=-\infty }^{L}2^{kq_{1}\alpha \left(
0\right) }\left( \sum \limits_{j=-\infty }^{k-1}\left \vert \lambda
_{j}\right \vert 2^{\left( \beta -n\delta _{2}\right) \left( k-j\right)
-\alpha \left( 0\right) j}\right) ^{q_{1}}
\end{equation*}%
\begin{eqnarray*}
&\lesssim &\sup \limits_{L\leq 0,L\in 
\mathbb{Z}
}2^{-L\lambda q_{1}}\sum \limits_{k=-\infty }^{L}2^{kq_{1}\alpha \left(
0\right) }\left( \sum \limits_{j=-\infty }^{k-1}\left \vert \lambda
_{j}\right \vert ^{q_{1}}2^{\left[ \left( \beta -n\delta _{2}\right) \left(
k-j\right) -\alpha \left( 0\right) j\right] \frac{q_{1}}{2}}\right) \\
&&\times \left( \sum \limits_{j=-\infty }^{k-1}2^{\left[ \left( \beta
-n\delta _{2}\right) \left( k-j\right) -\alpha \left( 0\right) j\right] 
\frac{q_{1}^{\prime }}{2}}\right) ^{\frac{q_{1}}{q_{1}^{\prime }}} \\
&\lesssim &\sup \limits_{L\leq 0,L\in 
\mathbb{Z}
}2^{-L\lambda q_{1}}\sum \limits_{k=-\infty }^{L}2^{kq_{1}\alpha \left(
0\right) }\left( \sum \limits_{j=-\infty }^{k-1}\left \vert \lambda
_{j}\right \vert ^{q_{1}}2^{\left[ \left( \beta -n\delta _{2}\right) \left(
k-j\right) -\alpha \left( 0\right) j\right] \frac{q_{1}}{2}}\right) \\
&\lesssim &\sup \limits_{L\leq 0,L\in 
\mathbb{Z}
}2^{-L\lambda q_{1}}\sum \limits_{j=-\infty }^{L}\left \vert \lambda
_{j}\right \vert ^{q_{1}}\sum \limits_{k=j+1}^{-1}2^{\left( \beta -n\delta
_{2}-\alpha \left( 0\right) \right) \left( k-j\right) \frac{q_{1}}{2}} \\
&\lesssim &\Lambda .
\end{eqnarray*}%
Thus, we have $F\lesssim \Lambda .$

Then we estimate $G$. To proceed, we consider%
\begin{eqnarray*}
G &=&\sum \limits_{k=-\infty }^{-1}2^{kq_{1}\alpha \left( 0\right) }\left
\Vert \left( I^{\beta }f\right) \chi _{k}\right \Vert _{L^{p_{2}(\cdot
)}}^{q_{1}} \\
&\leq &\sum \limits_{k=-\infty }^{-1}2^{kq_{1}\alpha \left( 0\right) }\left(
\sum \limits_{j=k}^{\infty }\left \vert \lambda _{j}\right \vert \left \Vert
\left( I^{\beta }a_{j}\right) \chi _{k}\right \Vert _{L^{p_{2}(\cdot
)}}\right) ^{q_{1}} \\
&&+\sum \limits_{k=-\infty }^{-1}2^{kq_{1}\alpha \left( 0\right) }\left(
\sum \limits_{j=-\infty }^{k-1}\left \vert \lambda _{j}\right \vert \left
\Vert \left( I^{\beta }a_{j}\right) \chi _{k}\right \Vert _{L^{p_{2}(\cdot
)}}\right) ^{q_{1}} \\
&:&=G_{1}+G_{2}.
\end{eqnarray*}%
Thus, we consider $G_{1}$ in two cases. Indeed, when $0<q_{1}\leq 1$, by (%
\ref{10}) and (\ref{11}), we obtain that%
\begin{eqnarray*}
G_{1} &=&\sum \limits_{k=-\infty }^{-1}2^{kq_{1}\alpha \left( 0\right)
}\left( \sum \limits_{j=k}^{\infty }\left \vert \lambda _{j}\right \vert
\left \Vert \left( I^{\beta }a_{j}\right) \chi _{k}\right \Vert
_{L^{p_{2}(\cdot )}}\right) ^{q_{1}} \\
&\lesssim &\sum \limits_{k=-\infty }^{-1}2^{kq_{1}\alpha \left( 0\right)
}\left( \sum \limits_{j=k}^{\infty }\left \vert \lambda _{j}\right \vert
2^{-j\alpha _{j}}\right) ^{q_{1}} \\
&\lesssim &\sum \limits_{k=-\infty }^{-1}2^{kq_{1}\alpha \left( 0\right)
}\left( \sum \limits_{j=k}^{-1}\left \vert \lambda _{j}\right \vert
^{q_{1}}2^{-j\alpha \left( 0\right) q_{1}}+\sum \limits_{j=0}^{\infty }\left
\vert \lambda _{j}\right \vert ^{q_{1}}2^{-j\alpha _{\infty }q_{1}}\right) \\
&\lesssim &\sum \limits_{k=-\infty }^{-1}\sum \limits_{j=k}^{-1}\left \vert
\lambda _{j}\right \vert ^{q_{1}}2^{\left( k-j\right) \alpha \left( 0\right)
q_{1}}+\sum \limits_{k=-\infty }^{-1}2^{k\alpha \left( 0\right) q_{1}}\sum
\limits_{j=0}^{\infty }\left \vert \lambda _{j}\right \vert
^{q_{1}}2^{-j\alpha _{\infty }q_{1}} \\
&\lesssim &\sum \limits_{j=-\infty }^{-1}\left \vert \lambda _{j}\right
\vert ^{q_{1}}\sum \limits_{k=\infty }^{j}2^{\left( k-j\right) \alpha \left(
0\right) q_{1}}+\sum \limits_{j=0}^{\infty }\left \vert \lambda _{j}\right
\vert ^{q_{1}}2^{-j\alpha _{\infty }q_{1}}\sum \limits_{k=-\infty
}^{-1}2^{k\alpha \left( 0\right) q_{1}}
\end{eqnarray*}%
\begin{eqnarray*}
&\lesssim &\sum \limits_{j=-\infty }^{-1}\left \vert \lambda _{j}\right
\vert ^{q_{1}}+\sum \limits_{j=0}^{\infty }\left \vert \lambda _{j}\right
\vert ^{q_{1}}2^{-j\alpha _{\infty }q_{1}}2^{-j\lambda q_{1}}\sum
\limits_{k=-\infty }^{-1}2^{k\alpha \left( 0\right) q_{1}} \\
&\lesssim &\Lambda +\Lambda \sum \limits_{i=-\infty }^{j}\left \vert \lambda
_{i}\right \vert ^{q_{1}}\sum \limits_{j=0}^{\infty }2^{\left( \lambda
-\alpha _{\infty }\right) q_{1}j}\sum \limits_{k=-\infty }^{j}2^{k\alpha
\left( 0\right) q_{1}} \\
&\lesssim &\Lambda \text{ }\left( \alpha _{\infty }>\lambda \right) .
\end{eqnarray*}%
When $1<q_{1}<\infty $, applying H\"{o}lder's inequality such that $\frac{1}{%
q_{1}}+\frac{1}{q_{1}^{\prime }}=1$ and (\ref{10}), we know that%
\begin{eqnarray*}
G_{1} &\lesssim &\sum \limits_{k=-\infty }^{-1}2^{kq_{1}\alpha \left(
0\right) }\left( \sum \limits_{j=k}^{\infty }\left \vert \lambda _{j}\right
\vert 2^{-j\alpha _{j}}\right) ^{q_{1}} \\
&\lesssim &\sum \limits_{k=-\infty }^{-1}\left( \sum \limits_{j=k}^{-1}\left
\vert \lambda _{j}\right \vert 2^{\left( k-j\right) \alpha \left( 0\right)
}\right) ^{q_{1}}+\sum \limits_{k=-\infty }^{-1}2^{kq_{1}\alpha \left(
0\right) }\left( \sum \limits_{j=0}^{\infty }\left \vert \lambda _{j}\right
\vert 2^{-j\alpha _{\infty }}\right) ^{q_{1}} \\
&\lesssim &\sum \limits_{k=-\infty }^{-1}\left( \sum \limits_{j=k}^{-1}\left
\vert \lambda _{j}\right \vert ^{q_{1}}2^{\left( k-j\right) \alpha \left(
0\right) \frac{q_{1}}{2}}\right) \left( \sum \limits_{j=k}^{-1}2^{\left(
k-j\right) \alpha \left( 0\right) \frac{q_{1}^{\prime }}{2}}\right) ^{\frac{%
q_{1}}{q_{1}^{\prime }}} \\
&&+\sum \limits_{k=-\infty }^{-1}2^{kq_{1}\alpha \left( 0\right) }\left(
\sum \limits_{j=0}^{\infty }\left \vert \lambda _{j}\right \vert
^{q_{1}}2^{-j\alpha _{\infty }\frac{q_{1}}{2}}\right) \left( \sum
\limits_{j=0}^{\infty }2^{-j\alpha _{\infty }\frac{q_{1}^{\prime }}{2}%
}\right) ^{\frac{q_{1}}{q_{1}^{\prime }}} \\
&\lesssim &\sum \limits_{k=-\infty }^{-1}\left \vert \lambda _{j}\right
\vert ^{q_{1}}\sum \limits_{k=-\infty }^{j}2^{\left( k-j\right) \alpha
\left( 0\right) \frac{q_{1}}{2}}+\sum \limits_{k=-\infty
}^{-1}2^{kq_{1}\alpha \left( 0\right) }\sum \limits_{j=0}^{\infty }\left
\vert \lambda _{j}\right \vert ^{q_{1}}2^{-j\alpha _{\infty }\frac{q_{1}}{2}}
\end{eqnarray*}%
\begin{eqnarray*}
&\lesssim &\sum \limits_{j=-\infty }^{-1}\left \vert \lambda _{j}\right
\vert ^{q_{1}}+\sum \limits_{j=0}^{\infty }2^{\left( \lambda -\frac{\alpha
_{\infty }}{2}\right) jq_{1}}2^{-j\lambda q_{1}}\sum \limits_{i=-\infty
}^{j}\left \vert \lambda _{i}\right \vert ^{q_{1}}\sum \limits_{k=-\infty
}^{-1}2^{kq_{1}\alpha \left( 0\right) } \\
&\lesssim &\Lambda +\Lambda \sum \limits_{j=0}^{\infty }2^{\left( \lambda -%
\frac{\alpha _{\infty }}{2}\right) jq_{1}}\sum \limits_{k=-\infty
}^{-1}2^{k\alpha \left( 0\right) q_{1}} \\
&\lesssim &\Lambda \text{ }\left( \alpha _{\infty }>\lambda \right) .
\end{eqnarray*}%
For $G_{2}$, when $0<q_{1}\leq 1$, by (\ref{12}), (\ref{11}) and the
assumption $\beta -n\delta _{2}<\alpha \left( 0\right) $, we obtain that%
\begin{eqnarray*}
G_{2} &=&\sum \limits_{k=-\infty }^{-1}2^{kq_{1}\alpha \left( 0\right)
}\left( \sum \limits_{j=-\infty }^{k-1}\left \vert \lambda _{j}\right \vert
\left \Vert \left( I^{\beta }a_{j}\right) \chi _{k}\right \Vert
_{L^{p_{2}(\cdot )}}\right) ^{q_{1}} \\
&\lesssim &\sum \limits_{k=-\infty }^{-1}2^{kq_{1}\alpha \left( 0\right)
}\left( \sum \limits_{j=-\infty }^{k-1}\left \vert \lambda _{j}\right \vert
^{q_{1}}2^{\left[ \left( \beta -n\delta _{2}\right) \left( k-j\right)
-\alpha \left( 0\right) j\right] q_{1}}\right) \\
&\lesssim &\sum \limits_{k=-\infty }^{-1}\left \vert \lambda _{j}\right
\vert ^{q_{1}}\sum \limits_{k=j+1}^{-1}2^{\left( \beta -n\delta _{2}-\alpha
\left( 0\right) \right) \left( k-j\right) q_{1}} \\
&\lesssim &\Lambda .
\end{eqnarray*}%
Now, let $1<q_{1}<\infty $. and $\frac{1}{q_{1}}+\frac{1}{q_{1}^{\prime }}=1$%
. By (\ref{12}), H\"{o}lder's inequality and the assumption $\beta -n\delta
_{2}<\alpha \left( 0\right) $,%
\begin{eqnarray*}
G_{2} &\lesssim &\sum \limits_{k=-\infty }^{-1}2^{kq_{1}\alpha \left(
0\right) }\left( \sum \limits_{j=-\infty }^{k-1}\left \vert \lambda
_{j}\right \vert 2^{\left( \beta -n\delta _{2}\right) \left( k-j\right)
-\alpha \left( 0\right) j}\right) ^{q_{1}} \\
&\lesssim &\sum \limits_{k=-\infty }^{-1}2^{kq_{1}\alpha \left( 0\right)
}\left( \sum \limits_{j=-\infty }^{k-1}\left \vert \lambda _{j}\right \vert
^{q_{1}}2^{\left[ \left( \beta -n\delta _{2}\right) \left( k-j\right)
-\alpha \left( 0\right) j\right] \frac{q_{1}}{2}}\right) \\
&&\times \left( \sum \limits_{j=-\infty }^{k-1}2^{\left[ \left( \beta
-n\delta _{2}\right) \left( k-j\right) -\alpha \left( 0\right) j\right] 
\frac{q_{1}^{\prime }}{2}}\right) ^{\frac{q_{1}}{q_{1}^{\prime }}} \\
&\lesssim &\sum \limits_{k=-\infty }^{-1}2^{kq_{1}\alpha \left( 0\right)
}\left( \sum \limits_{j=-\infty }^{k-1}\left \vert \lambda _{j}\right \vert
^{q_{1}}2^{\left[ \left( \beta -n\delta _{2}\right) \left( k-j\right)
-\alpha \left( 0\right) j\right] \frac{q_{1}}{2}}\right) \\
&\lesssim &\sum \limits_{j=-\infty }^{-1}\left \vert \lambda _{j}\right
\vert ^{q_{1}}\sum \limits_{k=j+1}^{-1}2^{\left( \beta -n\delta _{2}-\alpha
\left( 0\right) \right) \left( k-j\right) \frac{q_{1}}{2}} \\
&\lesssim &\Lambda .
\end{eqnarray*}%
Therefore, we have $G\lesssim \Lambda .$

Now, we turn to estimate $H$. We represent $H$ as%
\begin{eqnarray*}
H &=&\sup \limits_{L>0,L\in 
\mathbb{Z}
}2^{-L\lambda q_{1}}\left( \sum \limits_{k=0}^{L}2^{kq_{1}\alpha _{\infty
}}\left \Vert \left( I^{\beta }f\right) \chi _{k}\right \Vert _{L^{p_{2}(\cdot
)}}^{q_{1}}\right)  \\
&\lesssim &\sup \limits_{L>0,L\in 
\mathbb{Z}
}2^{-L\lambda q_{1}}\sum \limits_{k=0}^{L}2^{kq_{1}\alpha _{\infty }}\left(
\sum \limits_{j=k}^{\infty }\left \vert \lambda _{j}\right \vert \left \Vert
\left( I^{\beta }a_{j}\right) \chi _{k}\right \Vert _{L^{p_{2}(\cdot
)}}\right) ^{q_{1}} \\
&&+\sup \limits_{L>0,L\in 
\mathbb{Z}
}2^{-L\lambda q_{1}}\sum \limits_{k=0}^{L}2^{kq_{1}\alpha _{\infty }}\left(
\sum \limits_{j=-\infty }^{k-1}\left \vert \lambda _{j}\right \vert \left \Vert
\left( I^{\beta }a_{j}\right) \chi _{k}\right \Vert _{L^{p_{2}(\cdot
)}}\right) ^{q_{1}} \\
&:&=H_{1}+H_{2}.
\end{eqnarray*}%
To do so, we consider $H_{1}$ and $H_{2}$ into two cases $0<q_{1}\leq 1$ and 
$1<q_{1}<\infty $, as above. Similar to $F_{1}$, when $0<q_{1}\leq 1$, by
the boundedness of $I^{\beta }$ from $L^{p_{1}\left( \cdot \right) }$ to $%
L^{p_{2}\left( \cdot \right) }$ and using (\ref{10}) and (\ref{11}),%
\begin{eqnarray*}
H_{1} &=&\sup \limits_{L>0,L\in 
\mathbb{Z}
}2^{-L\lambda q_{1}}\sum \limits_{k=0}^{L}2^{kq_{1}\alpha _{\infty }}\left(
\sum \limits_{j=k}^{\infty }\left \vert \lambda _{j}\right \vert \left \Vert
\left( I^{\beta }a_{j}\right) \chi _{k}\right \Vert _{L^{p_{2}(\cdot
)}}\right) ^{q_{1}} \\
&\lesssim &\sup \limits_{L>0,L\in 
\mathbb{Z}
}2^{-L\lambda q_{1}}\sum \limits_{k=0}^{L}2^{kq_{1}\alpha _{\infty
}}\sum \limits_{j=k}^{\infty }\left \vert \lambda _{j}\right \vert
^{q_{1}}\left \Vert \left( I^{\beta }a_{j}\right) \chi _{k}\right \Vert
_{L^{p_{2}(\cdot )}}^{q_{1}} \\
&\lesssim &\sup \limits_{L>0,L\in 
\mathbb{Z}
}2^{-L\lambda q_{1}}\sum \limits_{k=0}^{L}2^{kq_{1}\alpha _{\infty
}}\sum \limits_{j=k}^{\infty }\left \vert \lambda _{j}\right \vert
^{q_{1}}2^{-j\alpha _{j}q_{1}}
\end{eqnarray*}%
\begin{eqnarray*}
&\lesssim &\sup \limits_{L>0,L\in 
\mathbb{Z}
}2^{-L\lambda q_{1}}\sum \limits_{k=0}^{L}2^{kq_{1}\alpha _{\infty
}}\sum \limits_{j=k}^{\infty }\left \vert \lambda _{j}\right \vert
^{q_{1}}2^{-j\alpha _{\infty }q_{1}} \\
&=&\sup \limits_{L>0,L\in 
\mathbb{Z}
}2^{-L\lambda q_{1}}\sum \limits_{j=0}^{L}\left \vert \lambda _{j}\right \vert
^{q_{1}}\sum \limits_{k=0}^{j}2^{\left( k-j\right) q_{1}\alpha _{\infty }} \\
&&+\sup \limits_{L>0,L\in 
\mathbb{Z}
}2^{-L\lambda q_{1}}\sum \limits_{j=L}^{\infty }\left \vert \lambda
_{j}\right \vert ^{q_{1}}\sum \limits_{k=0}^{L}2^{\left( k-j\right)
q_{1}\alpha _{\infty }} \\
&\lesssim &\sup \limits_{L>0,L\in 
\mathbb{Z}
}2^{-L\lambda q_{1}}\sum \limits_{j=0}^{L}\left \vert \lambda _{j}\right \vert
^{q_{1}} \\
&&+\sup \limits_{L>0,L\in 
\mathbb{Z}
}\sum \limits_{j=L}^{\infty }2^{\left( jq_{1}\lambda -Lq_{1}\lambda \right)
}2^{-j\lambda q_{1}}\dsum \limits_{i=-\infty }^{j}\left \vert \lambda
_{i}\right \vert ^{q_{1}}\sum \limits_{k=0}^{L}2^{\left( k-j\right)
q_{1}\alpha _{\infty }} \\
&\lesssim &\Lambda +\Lambda \sup \limits_{L>0,L\in 
\mathbb{Z}
}\sum \limits_{j=L}^{\infty }2^{\left( j-L\right) \lambda q_{1}}2^{\left(
L-j\right) q_{1}\alpha _{\infty }} \\
&\lesssim &\Lambda +\Lambda \sup \limits_{L>0,L\in 
\mathbb{Z}
}\sum \limits_{j=L}^{\infty }2^{\left( j-L\right) q_{1}\left( \lambda -\alpha
_{\infty }\right) } \\
&\lesssim &\Lambda \text{ }\left( \alpha _{\infty }>\lambda \right) .
\end{eqnarray*}%
On the other hand, when $1<q_{1}<\infty $, by the boundedness of $I^{\beta }$
from $L^{p_{1}\left( \cdot \right) }$ to $L^{p_{2}\left( \cdot \right) }$
and using (\ref{10}) and H\"{o}lder's inequality, we have%
\begin{eqnarray*}
H_{1} &\lesssim &\sup \limits_{L>0,L\in 
\mathbb{Z}
}2^{-L\lambda q_{1}}\sum \limits_{k=0}^{L}2^{kq_{1}\alpha _{\infty }}\left(
\sum \limits_{j=k}^{\infty }\left \vert \lambda _{j}\right \vert
^{q_{1}}\left \Vert \left( I^{\beta }a_{j}\right) \chi _{k}\right \Vert
_{L^{p_{2}(\cdot )}}^{\frac{q_{1}}{2}}\right)  \\
&&\times \left( \sum \limits_{j=k}^{\infty }\left \Vert \left( I^{\beta
}a_{j}\right) \chi _{k}\right \Vert _{L^{p_{2}(\cdot )}}^{\frac{q_{1}^{\prime
}}{2}}\right) ^{\frac{q_{1}}{q_{1}^{\prime }}} \\
&\lesssim &\sup \limits_{L>0,L\in 
\mathbb{Z}
}2^{-L\lambda q_{1}}\sum \limits_{k=0}^{L}2^{kq_{1}\alpha _{\infty }}\left(
\sum \limits_{j=k}^{\infty }\left \vert \lambda _{j}\right \vert
^{q_{1}}\left \Vert a_{j}\right \Vert _{L^{p_{1}(\cdot )}}^{\frac{q_{1}}{2}%
}\right)  \\
&&\times \left( \sum \limits_{j=k}^{\infty }\left \Vert a_{j}\right \Vert
_{L^{p_{1}(\cdot )}}^{\frac{q_{1}^{\prime }}{2}}\right) ^{\frac{q_{1}}{%
q_{1}^{\prime }}} \\
&\lesssim &\sup \limits_{L>0,L\in 
\mathbb{Z}
}2^{-L\lambda q_{1}}\sum \limits_{k=0}^{L}2^{kq_{1}\alpha _{\infty }}\left(
\sum \limits_{j=k}^{\infty }\left \vert \lambda _{j}\right \vert
^{q_{1}}\left \vert B_{j}\right \vert ^{-\frac{\alpha _{j}q_{1}}{\left(
2n\right) }}\right)  \\
&&\times \left( \sum \limits_{j=k}^{\infty }\left \vert B_{j}\right \vert ^{-%
\frac{\alpha _{j}q_{1}^{\prime }}{\left( 2n\right) }}\right) ^{\frac{q_{1}}{%
q_{1}^{\prime }}} \\
&\lesssim &\sup \limits_{L>0,L\in 
\mathbb{Z}
}2^{-L\lambda q_{1}}\sum \limits_{k=0}^{L}2^{k\frac{q_{1}}{2}\alpha _{\infty
}}\left( \sum \limits_{j=k}^{\infty }\left \vert \lambda _{j}\right \vert
^{q_{1}}\left \vert B_{j}\right \vert ^{-\frac{\alpha _{j}q_{1}}{\left(
2n\right) }}\right) 
\end{eqnarray*}%
\begin{eqnarray*}
&=&\sup \limits_{L>0,L\in 
\mathbb{Z}
}2^{-L\lambda q_{1}}\sum \limits_{j=0}^{L}\left \vert \lambda _{j}\right \vert
^{q_{1}}\sum \limits_{k=0}^{j}2^{\left( k-j\right) \frac{q_{1}}{2}\alpha
_{\infty }} \\
&&+\sup \limits_{L>0,L\in 
\mathbb{Z}
}2^{-L\lambda q_{1}}\sum \limits_{j=L}^{\infty }\left \vert \lambda
_{j}\right \vert ^{q_{1}}\sum \limits_{k=0}^{L}2^{\left( k-j\right) \frac{q_{1}%
}{2}\alpha _{\infty }} \\
&\lesssim &\sup \limits_{L>0,L\in 
\mathbb{Z}
}2^{-L\lambda q_{1}}\sum \limits_{j=0}^{L}\left \vert \lambda _{j}\right \vert
^{q_{1}} \\
&&+\sup \limits_{L>0,L\in 
\mathbb{Z}
}\sum \limits_{j=L}^{\infty }2^{\left( jq_{1}\lambda -Lq_{1}\lambda \right)
}2^{-j\lambda q_{1}}\dsum \limits_{i=-\infty }^{j}\left \vert \lambda
_{i}\right \vert ^{q_{1}}\sum \limits_{k=0}^{L}2^{\left( k-j\right) \frac{q_{1}%
}{2}\alpha _{\infty }} \\
&\lesssim &\Lambda +\Lambda \sup \limits_{L>0,L\in 
\mathbb{Z}
}\sum \limits_{j=L}^{\infty }2^{\left( j-L\right) \lambda q_{1}}2^{\left(
L-j\right) \frac{q_{1}}{2}\alpha _{\infty }} \\
&\lesssim &\Lambda +\Lambda \sup \limits_{L>0,L\in 
\mathbb{Z}
}\sum \limits_{j=L}^{\infty }2^{\left( j-L\right) q_{1}\left( \lambda -\frac{%
\alpha _{\infty }}{2}\right) } \\
&\lesssim &\Lambda \text{ }\left( \alpha _{\infty }>2\lambda \right) .
\end{eqnarray*}%
Finally, we consider the term $H_{2}$ into two cases $0<q_{1}\leq 1$ and $%
1<q_{1}<\infty $.

When $0<q_{1}\leq 1$, by (\ref{0}), (\ref{11}) and the assumptions that $%
\beta -n\delta _{2}<\alpha \left( 0\right) $ and $\beta -n\delta _{2}<\alpha
_{\infty }$, we get%
\begin{eqnarray*}
H_{2} &=&\sup \limits_{L>0,L\in 
\mathbb{Z}
}2^{-L\lambda q_{1}}\sum \limits_{k=0}^{L}2^{kq_{1}\alpha _{\infty }}\left(
\sum \limits_{j=-\infty }^{k-1}\left \vert \lambda _{j}\right \vert \left \Vert
\left( I^{\beta }a_{j}\right) \chi _{k}\right \Vert _{L^{p_{2}(\cdot
)}}\right) ^{q_{1}} \\
&\lesssim &\sup \limits_{L>0,L\in 
\mathbb{Z}
}2^{-L\lambda q_{1}}\sum \limits_{k=0}^{L}2^{kq_{1}\alpha _{\infty }}\left(
\sum \limits_{j=-\infty }^{k-1}\left \vert \lambda _{j}\right \vert ^{q_{1}}2^{%
\left[ \left( \beta -n\delta _{1}\right) \left( j-k\right) -\alpha _{j}j%
\right] q_{1}}\right)  \\
&=&\sup \limits_{L>0,L\in 
\mathbb{Z}
}2^{-L\lambda q_{1}}\sum \limits_{k=0}^{L}2^{kq_{1}\alpha _{\infty }}\left(
\sum \limits_{j=-\infty }^{-1}\left \vert \lambda _{j}\right \vert ^{q_{1}}2^{%
\left[ \left( \beta -n\delta _{1}\right) \left( j-k\right) -\alpha \left(
0\right) j\right] q_{1}}\right)  \\
&&+\sup \limits_{L>0,L\in 
\mathbb{Z}
}2^{-L\lambda q_{1}}\sum \limits_{k=0}^{L}2^{kq_{1}\alpha _{\infty }}\left(
\sum \limits_{j=0}^{k-1}\left \vert \lambda _{j}\right \vert ^{q_{1}}2^{\left[
\left( \beta -n\delta _{1}\right) \left( j-k\right) -\alpha _{\infty }j%
\right] q_{1}}\right)  \\
&\lesssim &\sup \limits_{L>0,L\in 
\mathbb{Z}
}2^{-L\lambda q_{1}}\sum \limits_{k=0}^{L}2^{kq_{1}\left[ \alpha _{\infty
}-\left( \beta -n\delta _{1}\right) \right] }\sum \limits_{j=-\infty
}^{-1}\left \vert \lambda _{j}\right \vert ^{q_{1}}2^{\left[ \beta -n\delta
_{1}-\alpha \left( 0\right) \right] jq_{1}} \\
&&+\sup \limits_{L>0,L\in 
\mathbb{Z}
}2^{-L\lambda q_{1}}\sum \limits_{k=0}^{L}\left \vert \lambda _{j}\right \vert
^{q_{1}}\sum \limits_{k=j+1}^{\infty }2^{\left[ \left( \beta -n\delta
_{1}-\alpha _{\infty }\right) \left( j-k\right) \right] q_{1}} \\
&\lesssim &\sup \limits_{L>0,L\in 
\mathbb{Z}
}2^{-L\lambda q_{1}}\sum \limits_{j=-\infty }^{-1}\left \vert \lambda
_{j}\right \vert ^{q_{1}}+\sup \limits_{L>0,L\in 
\mathbb{Z}
}2^{-L\lambda q_{1}}\sum \limits_{j=0}^{L-1}\left \vert \lambda
_{j}\right \vert ^{q_{1}} \\
&\lesssim &\Lambda .
\end{eqnarray*}%
Now, let $1<q_{1}<\infty $. and $\frac{1}{q_{1}}+\frac{1}{q_{1}^{\prime }}=1$%
. By (\ref{0}), using H\"{o}lder's inequality and the conditions $\beta
-n\delta _{2}<\alpha \left( 0\right) $ and $\beta -n\delta _{2}<\alpha
_{\infty }$, we obtain that%
\begin{eqnarray*}
H_{2} &\lesssim &\sup \limits_{L>0,L\in 
\mathbb{Z}
}2^{-L\lambda q_{1}}\sum \limits_{k=0}^{L}2^{kq_{1}\alpha _{\infty }}\left(
\sum \limits_{j=-\infty }^{k-1}\left \vert \lambda _{j}\right \vert 2^{\left(
\beta -n\delta _{1}\right) \left( j-k\right) -\alpha _{j}j}\right) ^{q_{1}}
\\
&\lesssim &\sup \limits_{L>0,L\in 
\mathbb{Z}
}2^{-L\lambda q_{1}}\sum \limits_{k=0}^{L}2^{kq_{1}\alpha _{\infty }}\left(
\sum \limits_{j=-\infty }^{-1}\left \vert \lambda _{j}\right \vert 2^{\left(
\beta -n\delta _{1}\right) \left( j-k\right) -\alpha \left( 0\right)
j}\right) ^{q_{1}} \\
&&+\sup \limits_{L>0,L\in 
\mathbb{Z}
}2^{-L\lambda q_{1}}\sum \limits_{k=0}^{L}2^{kq_{1}\alpha _{\infty }}\left(
\sum \limits_{j=0}^{k-1}\left \vert \lambda _{j}\right \vert 2^{\left( \beta
-n\delta _{1}\right) \left( j-k\right) -\alpha _{\infty }j}\right) ^{q_{1}}
\\
&\lesssim &\sup \limits_{L>0,L\in 
\mathbb{Z}
}2^{-L\lambda q_{1}}\sum \limits_{k=0}^{L}2^{kq_{1}\left( \alpha _{\infty
}+\beta -n\delta _{1}\right) }\left( \sum \limits_{j=-\infty }^{-1}\left \vert
\lambda _{j}\right \vert 2^{-j\left[ \left( \beta -n\delta _{1}\right)
+\alpha \left( 0\right) \right] }\right) ^{q_{1}} \\
&&+\sup \limits_{L>0,L\in 
\mathbb{Z}
}2^{-L\lambda q_{1}}\sum \limits_{k=0}^{L}\left(
\sum \limits_{j=0}^{k-1}\left \vert \lambda _{j}\right \vert 2^{\left( \beta
-n\delta _{1}-\alpha _{\infty }\right) \left( j-k\right) }\right) ^{q_{1}} \\
&\lesssim &\left( \sup \limits_{L>0,L\in 
\mathbb{Z}
}2^{-L\lambda q_{1}}\sum \limits_{j=-\infty }^{-1}\left \vert \lambda
_{j}\right \vert ^{q_{1}}2^{\left[ \left( \beta -n\delta _{1}\right) -\alpha
\left( 0\right) \right] j\frac{q_{1}}{2}}\right)  \\
&&\times \left( \sum \limits_{j=-\infty }^{-1}2^{\left[ \left( \beta -n\delta
_{1}\right) -\alpha \left( 0\right) \right] j\frac{q_{1}^{\prime }}{2}%
}\right) ^{\frac{q_{1}}{q_{1}^{\prime }}} \\
&&+\sup \limits_{L>0,L\in 
\mathbb{Z}
}2^{-L\lambda q_{1}}\sum \limits_{k=0}^{L}\left(
\sum \limits_{j=0}^{k-1}\left \vert \lambda _{j}\right \vert ^{q_{1}}2^{\left[
\left( \beta -n\delta _{1}-\alpha _{\infty }\right) \left( j-k\right) \right]
\frac{q_{1}}{2}}\right)  \\
&&\times \left( \sum \limits_{j=0}^{k-1}2^{\left[ \left( \beta -n\delta
_{1}-\alpha _{\infty }\right) \left( j-k\right) \right] \frac{q_{1}^{\prime }%
}{2}}\right) ^{\frac{q_{1}}{q_{1}^{\prime }}} \\
&\lesssim &\sup \limits_{L>0,L\in 
\mathbb{Z}
}2^{-L\lambda q_{1}}\sum \limits_{j=-\infty }^{-1}\left \vert \lambda
_{j}\right \vert ^{q_{1}}2^{\left[ \left( \beta -n\delta _{1}\right) -\alpha
\left( 0\right) \right] j\frac{q_{1}}{2}} \\
&&+\sup \limits_{L>0,L\in 
\mathbb{Z}
}2^{-L\lambda q_{1}}\sum \limits_{k=0}^{L}\sum \limits_{j=0}^{k-1}\left \vert
\lambda _{j}\right \vert ^{q_{1}}2^{\left[ \left( \beta -n\delta _{1}-\alpha
_{\infty }\right) \left( j-k\right) \right] \frac{q_{1}}{2}} \\
&\lesssim &\sup \limits_{L>0,L\in 
\mathbb{Z}
}2^{-L\lambda q_{1}}\sum \limits_{j=-\infty }^{-1}\left \vert \lambda
_{j}\right \vert ^{q_{1}} \\
&&+\sup \limits_{L>0,L\in 
\mathbb{Z}
}2^{-L\lambda q_{1}}\sum \limits_{j=0}^{L-1}\left \vert \lambda
_{j}\right \vert ^{q_{1}}\sum \limits_{k=j+1}^{L}2^{\left[ \left( \beta
-n\delta _{1}-\alpha _{\infty }\right) \left( j-k\right) \right] \frac{q_{1}%
}{2}} \\
&\lesssim &\sup \limits_{L>0,L\in 
\mathbb{Z}
}2^{-L\lambda q_{1}}\sum \limits_{j=-\infty }^{-1}\left \vert \lambda
_{j}\right \vert ^{q_{1}}+\sup \limits_{L>0,L\in 
\mathbb{Z}
}2^{-L\lambda q_{1}}\sum \limits_{j=0}^{L-1}\left \vert \lambda
_{j}\right \vert ^{q_{1}} \\
&\lesssim &\Lambda .
\end{eqnarray*}%
Hence, we have $H\lesssim \Lambda .$ By combining the above inequalities for 
$F,G$ and $H$, we obtain (\ref{100}).
\end{proof}

\end{document}